\documentclass[11pt]{amsart}
\usepackage{latexsym,amsfonts,amssymb,graphicx,amsmath}
\usepackage{epstopdf}
\usepackage{amssymb}
\usepackage{enumerate}
\usepackage{enumitem}
\usepackage{booktabs}
\usepackage{stmaryrd}
\usepackage{color}
\usepackage{sidecap}
\usepackage[font=small,labelfont=bf]{caption}
\numberwithin{equation}{section}

\newcommand{\n}[1]{\boldsymbol{#1}}

\newcommand{\TG}{ T_\Gamma}
\newcommand{\ThG}{\mathcal{T}_h^\Gamma}

\newcommand{\EhGm}{\mathcal{E}_h^{\Gamma,-}}
\newcommand{\EhGp}{\mathcal{E}_h^{\Gamma,+}}

\newcommand{\revblue}[1]{\textcolor{black}{{#1}}}

\newcommand{\ra}[1]{\renewcommand{\arraystretch}{#1}}

\newcommand{\Omclosed}{\overline{\Omega}^{-}}

\newcommand{\Opclosed}{\overline{\Omega}^{+}}

\newtheorem{theorem}{Theorem}
\newtheorem{lemma}{Lemma}
\newtheorem{proposition}{Proposition}

\theoremstyle{definition}

\newtheorem{corollary}{Corollary}
\newtheorem{remark}{Remark}

\begin{document}

\title[Nitsche method for high contrast interface problems]{Robust flux error estimation of \revblue{an unfitted Nitsche} method for high-contrast interface problems}

\author[E. Burman]{Erik Burman \textsuperscript{1}}
\address{\textsuperscript{1} Department of Mathematics, University College London, London, UK}
\email{e.burman@ucl.ac.uk}
\author[J. Guzm\'an]{Johnny Guzm\'an\textsuperscript{2}}
\address{\textsuperscript{2} Division of Applied Mathematics, Brown University, Providence, RI 02912, USA}
\email{johnny\_guzman@brown.edu}
\author[M. A. S\'anchez]{Manuel A. S\'anchez\textsuperscript{3}}
\address{\textsuperscript{3} School of Mathematics, University of Minnesota, Minneapolis, MN 55455, USA}
\email{sanchez@umn.edu}
\author[M. Sarkis]{Marcus Sarkis\textsuperscript{4}}
\address{\textsuperscript{4}Department of Mathematical Sciences at Worcester Polytechnic Institute, 100 Institute Road, Worcester, MA 01609, USA }
\email{msarkis@wpi.edu}


\maketitle

\begin{abstract}
{\revblue{We prove an optimal error estimate for the flux variable for a stabilized unfitted Nitsche finite element method applied to an elliptic interface problem with discontinuous constant coefficients. Our result shows explicitly that this error estimate is totally independent of the diffusion coefficients}.}
{Interface problems; high-contrast; \revblue{unfitted Nitsche} method; finite elements.}
\end{abstract}

\section{Introduction}\label{section1}
In this paper we study the error estimation of \revblue{an unfitted Nitsche} finite element method for the following elliptic interface problem with discontinuous constant coefficients: Let $\Omega\subset \mathbb{R}^2$ be an open polygonal domain with an  immersed smooth interface $\Gamma$, such that $\overline{\Omega}=\revblue{\overline{\Omega^{-}} \cup \,\overline{\Omega^{+}}}$, and $\Gamma$ encloses either $\Omega^{-}$ or $\Omega^+$. Consider the problem
\begin{subequations}\label{Problem}
\begin{alignat}{3}
    -\nabla\cdot(\rho^{\pm} \nabla u^{\pm}) &= f^{\pm} \qquad &\mbox{in } &\Omega^{\pm}, \label{Problem:a} \\
                                           u^{\pm} &=0            &\mbox{on }&\partial \Omega^{\pm}\backslash\Gamma, \label{Problem:b}\\
                            \left[u\right]&=0            &\mbox{on }&\Gamma, \label{Problem:c} \\
  \left[\rho \nabla  u\cdot \n{n}  \right]&=0            &\mbox{on }&\Gamma. \label{Problem:d}
\end{alignat}
\end{subequations}
The jumps on the interface $\Gamma$ are defined as
\begin{equation}\label{jumps}
\left[\rho \nabla u\cdot \n{n}  \right]=\rho^-\nabla u^{-} \cdot \n{n}^{-}+ \rho^+\nabla u^{+} \cdot \n{n}^+ \quad\mbox{and}\quad \left[u\right]=u^{+}-u^{-},
\end{equation}
where $u^{\pm}=u|_{\Omega^{\pm}}$ and $\n{n}^{\pm}$ is the unit outward pointing normal to $\Omega^{\pm}$. We furthermore assume that the diffusion coefficients $\rho^{+}\ge \rho^{-}>0$ are constant.

There have been several numerical methods for problem \eqref{Problem}. See
for example \cite{MR0277119}, \cite{2013arXiv1311.4178X}, \cite{MR2837483}, \cite{MR2002258}, \cite{MR1941489}, \cite{MR2377272}, \cite{MR2571349}, \cite{MR2684351}, \cite{MR2738930}, \cite{MR2566075}, \cite{MR2740492}, \cite{MR2820966}, \cite{MR3051411}, \cite{MR2981355}, \cite{MR3138107}, \cite{MR3218337}, \cite{MR3268662}, \cite{GSS15}.  The method we will consider below uses meshes that are not necessarily aligned with the mesh (i.e. unfitted meshes). There are several papers dealing with methods (see \cite{MR3218337}, \cite{MR2981355}, \cite{MR2571349}, \cite{MR2738930}, \cite{MR2684351}, \cite{MR2377272}, \cite{GSS15}, \cite{MR2740492}, \cite{MR3268662}, \cite{MR2820966}) using unfitted meshes. One of the advantages of using unfitted meshes is the fact that re-meshing is not required for problems where the interface is moving. Nevertheless, the majority of the unfitted methods do not address the analysis of high contrast problems. Some exceptions are found \cite{MR2684351}, \cite{MR2257119},\cite{MR2820966}, \cite{MR3051411}, \cite{GSS15}. In particular in \cite{MR2684351}, energy error estimates independent of the contrast of the coefficients (i.e. $\rho^{+}\revblue{/}\rho^{-}$) were proved. However, the estimates were not completely independent of the coefficients, a factor of $1\revblue{/}\sqrt{\rho^{-}}$ was present in the right-hand side. 
More recently, in \cite{GSS15} an interface finite element method was designed and
certain error estimates independent of coefficients contrast were proved on problems with smooth interfaces. Specifically, the error estimate achieved in \cite{GSS15} for the energy error was of the form
\begin{equation}\label{esrho}
\|\sqrt{\rho} \nabla(u-u_h)\|_{L^2(\Omega)} \le \frac{C}{\sqrt{\rho^-}} \, h \|f\|_{L^2(\Omega)},
\end{equation}
were we also observe a factor of $1\revblue{/}\sqrt{\rho^{-}}$ in the estimate.
One of the key ingredients in \cite{GSS15} was to add a stabilization term that penalized the jump of the gradients across edges of the triangulation. This idea was borrowed from the stabilized Nitsche's methods developed by Burman and co-authors; see for example \cite{MR3051411}. Here in this paper
we analyze a variant of the method introduced in \cite{MR3051411} and prove the following error estimate
totally independent of contrast
\begin{equation}\label{erho}
\|\rho \nabla(u-u_h)\|_{L^2(\Omega)} \le C \, h \|f\|_{L^2(\Omega)},
\end{equation}
where constant $C$ is independent of $\rho^{\pm}$. It is important to note that this estimate is for the
flux error $\rho \nabla(u-u_h)$. The previous analysis in \cite{GSS15,MR2684351,MR3051411} used energy arguments to
establish error estimates for the energy error $\sqrt{\rho} \nabla(u-u_h)$, resulting in the dependence of $1\revblue{/} \rho^{-}$. Notice that a simple application of estimate \eqref{esrho} will give
\begin{equation*}
\|\rho^+ \nabla(u-u_h)^+\|_{L^2(\Omega^+)} \le \frac{\sqrt{\rho^+}}{\sqrt{\rho^-}} \, h \|f\|_{L^2(\Omega)},
\end{equation*}
Hence, we see that our result \eqref{erho} here is much sharper for this quantity.  The main ingredient of the analysis is the use of a
discrete extension result from $\Omega^+$ to all of $\Omega$. We note that this technique can be extended
to conforming finite element discretizations and it opens the possibility to establish sharper results
also for other discretizations as well (in particular the method studied in \cite{GSS15}).

The paper is organized as follows. In the next section we describe the Nitsche's finite element method. In Section \ref{section3} we provide an error estimate based on an energy argument. In Section \ref{section4} we improve the main result, obtaining an error estimate independent of the contrast for the diffusion coefficients for the flux. In the following section we discuss extension of the method and results for: interface problem with non homogeneous jumps, and the three dimensional problem. In Section \ref{sectionnumerical} we present numerical results that validate the theoretical results. We conclude with an appendix that contains proofs of some crucial lemmas.

\section{Finite element method}\label{section2}

\subsection{Preliminaries}
Let $\{\mathcal{T}_h\}_{h>0}$ be an admissible family of triangulations of $\Omega$. We adopt the convention that elements $T$ and element edges $e$ are open sets. We use over-line symbol to refer to their closure. For each triangular element $T\in \mathcal{T}_h$, let $h_T$ denotes its diameter and define the global parameter of the triangulation by $h = \max_{T} h_T$. We assume that $\mathcal{T}_h$ is shape regular, i.e. there exists $\kappa>0$ such that for every $T\in \mathcal{T}_h$  the radius $\rho_T$ of its inscribed circle satisfies
\begin{equation}\label{shaperegularity}
\rho_T>h_T/\kappa.
\end{equation}

The set of elements cutting the interface $\Gamma$, and restricted to $\Omega^+$ and $\Omega^-$ are also of interest. They are defined by:
\begin{align*}
\mathcal{T}_h^{\pm}    &:= \{T\in \mathcal{T}_h: T\cap \Omega^{\pm} \neq \emptyset\},\\
\mathcal{T}_h^{\Gamma} &:= \{T\in \mathcal{T}^{-}_h: \overline{T}\cap \Gamma \neq \emptyset\}.
\end{align*}
In particular for $T\in \mathcal{T}_h^{\Gamma}$ we denote $T_{\Gamma} =\overline{T}\cap \Gamma$. Observe that the definition of $\mathcal{T}_h^{\Gamma}$ guarantees that $\sum_{T\in\mathcal{T}_h^{\Gamma}}|T_\Gamma| = |\Gamma|$. Under these definitions we define the discrete domains
\begin{equation*}
\Omega_h^{\pm} := \mbox{Int} \Big(\bigcup_{T\in\mathcal{T}_h^{\pm}} \overline{T}\Big).
\end{equation*}
See Figure \ref{figOmegas} for an illustration of these definitions. The set of all the edges of $\mathcal{T}_h^\Gamma$ restricted to the interior of $\Omega_h^{+}$ and $\Omega_h^{-}$ is also considered
\begin{align*}
\mathcal{E}_h^{\Gamma,\pm} &:= \{ e =  \mbox{Int} (\partial T_1\cap \partial T_2): T_1, T_2\in \mathcal{T}_h^{\pm}, \mbox{ and } T_1\cap\Gamma\neq\emptyset \mbox{ or } T_2\cap\Gamma\neq\emptyset\}.
\end{align*}
\begin{figure}
\begin{center}
\includegraphics[scale=.07]{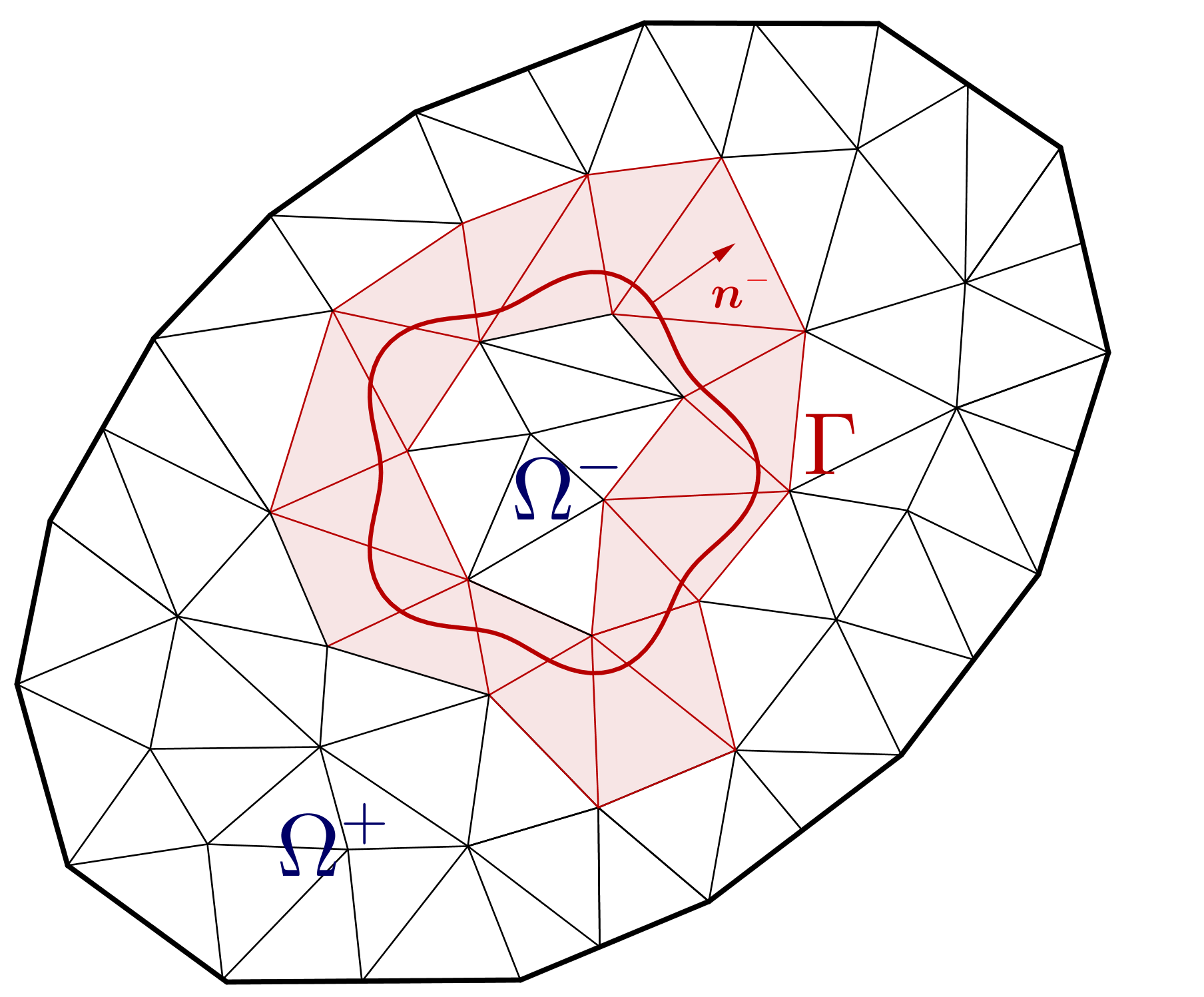}\includegraphics[scale=.07]{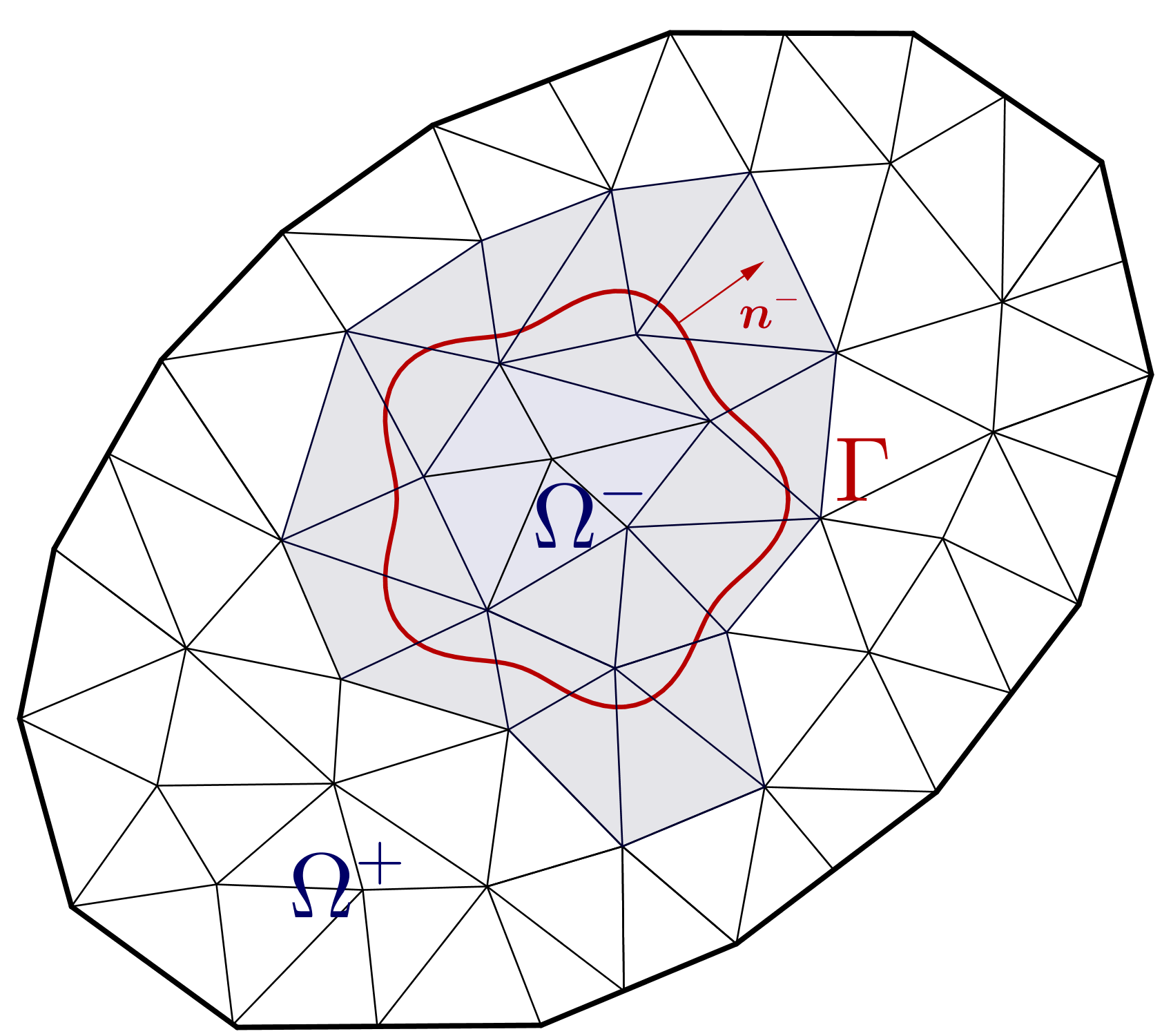}\includegraphics[scale=.07]{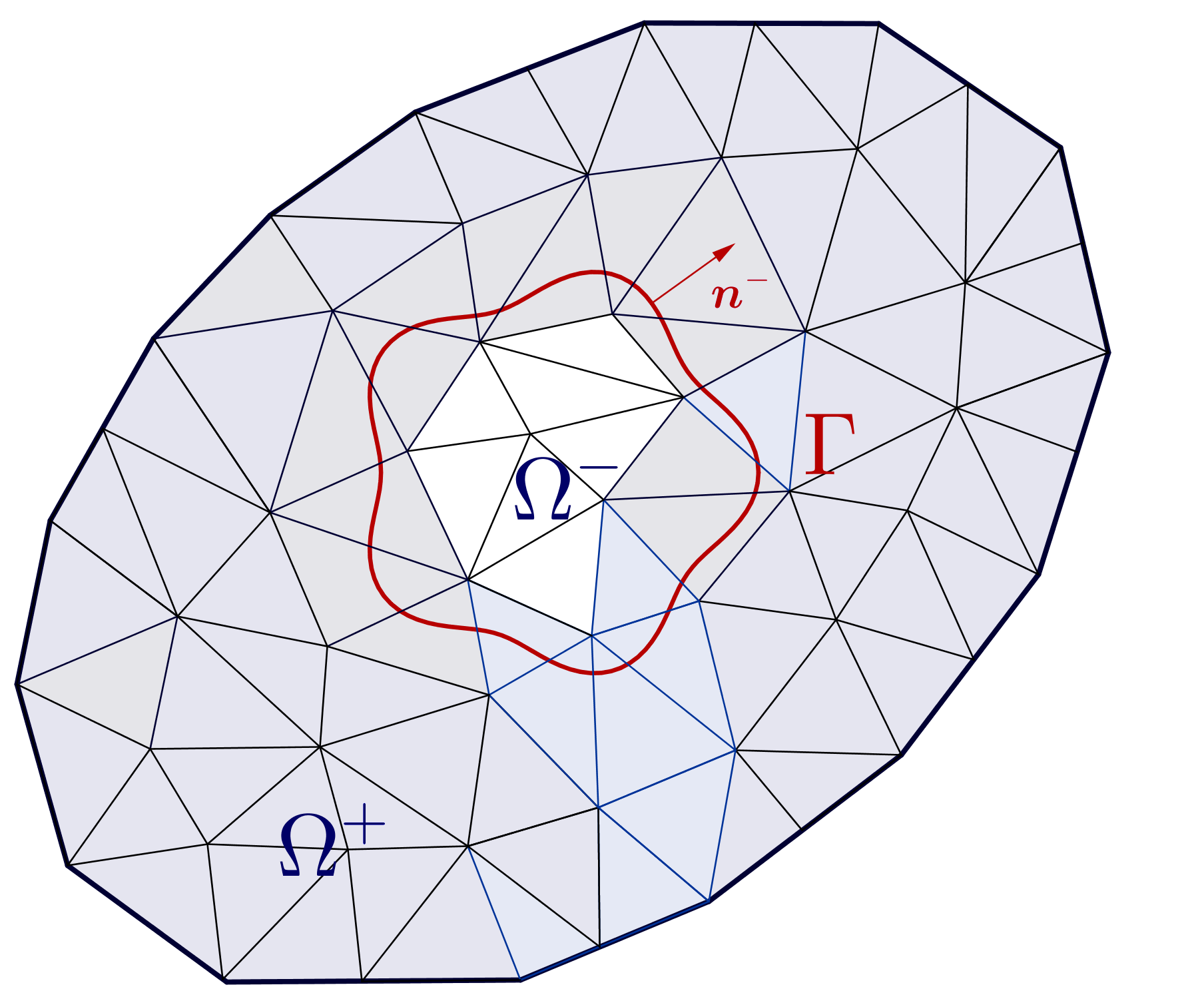}\\
\end{center}
\footnotesize
\ra{1.1}
\caption{Illustration of the definition of set $\mathcal{T}_h^{\Gamma}$ and discrete sub-domains $\Omega_h^{\pm}$. Left figure: elements in $\mathcal{T}_h^{\Gamma}$  (red transparent). Center figure: sub-domain $\Omega_h^{-}$ (blue transparent). Right figure: sub-domain $\Omega_h^{+}$ (blue transparent).}
\label{figOmegas}
\end{figure}

Standard finite element notation for the normal jumps is introduced as follows: for a piecewise smooth function $v$ with support in $\mathcal{T}_h$, the jump of its gradient across an interior edge $e=\mathrm{Int}(\partial T_1 \cap \partial T_2)$ is defined by
\begin{equation*}
\llbracket \nabla v \rrbracket= \nabla v|_{T_1} \cdot \n{n}_1+ \nabla v|_{T_2}\cdot  \n{n}_2,
\end{equation*}
where $\n{n}_1$ and $\n{n}_2$ are the unit normal vectors to $e$, pointing outwards to $T_1$ and $T_2$, respectively.
\subsection{Stabilized unfitted Nitsche method}
In this section we introduce a slightly simplified version of the Nitsche method for high-contrast interface problems by \cite{MR3051411}, Section 3.3.
We begin \revblue{by} denoting the standard finite element space of continuous piecewise linear polynomials with support in $\Omega_h^{\pm}$ by:
\begin{equation*}
V_h^{\pm} =\{ v \in \mathcal{C}(\Omega^{\pm}_h): v|_T \in \mathbb{P}^1(T), \forall T \in \mathcal{T}_h^{\pm}, \mbox{ and } v|_{\partial\Omega^{\pm}\backslash \Gamma } \equiv 0 \}.
\end{equation*}
The finite element space is defined by means of the restrictions of piecewise linear functions to $\Omega_h^+$ and $\Omega_h^-$, i.e.
\begin{equation*}
V_h= V_{h}^{-}\times V_h^{+}.
\end{equation*}
The jumps across the interface of a function $v\in V_h$ are defined as in \eqref{jumps}.

We now consider a finite element method based on: the weak formulation of problem \eqref{Problem}, penalty terms of the jump across the interface, and stabilization terms on edges in $\mathcal{E}_h^{\Gamma,\pm}$. Find $u_h=(u_h^{-},u_h^{+}) \in V_h$, such that:
\begin{equation}\label{fem}
a_h(u_h, v)=(f^{+},v^{+})_{\Omega^{+}}+(f^{-},v^{-})_{\Omega^{-}} , \quad \text{ for all } v \in V_h,
\end{equation}
where $(\cdot,\cdot)_{\Omega^{\pm}}$ denotes the $L^2$ product in $\Omega^{\pm}$ and $a_h(\cdot,\cdot)$ is a bilinear form defined by
\begin{alignat}{1}
&a_h(u_h, v)= \int_{\Omega^{+}} \rho^+ \nabla u_h^+ \cdot \nabla v^+dx  + \int_{\Omega^{-}} \rho^- \nabla u_h^- \cdot \nabla v^-dx  \label{a_h}\\
  & + \int_{\Gamma} \left(\rho^- \nabla v^- \cdot \n{n}^-  [u_h]+  \rho^- \nabla u_h^- \cdot \n{n}^-  [v]\right)ds + \sum_{T\in\ThG}\frac{\gamma}{h_T} \rho^-\int_{T_\Gamma} [u_h][v] ds\nonumber\\
& + \gamma_g^-\sum_{e \in \mathcal{E}_{h}^{\Gamma,-}}  |e| \int_{e} \rho^- \llbracket \nabla v^- \rrbracket\llbracket \nabla u_h^- \rrbracket \,ds +\gamma_g^+\sum_{e \in \mathcal{E}_{h}^{\Gamma,+}}  |e|\int_{e} \rho^+ \llbracket \nabla v^+ \rrbracket  \llbracket \nabla u_h^+ \rrbracket \,ds, \nonumber
\end{alignat}
where $\gamma, \gamma_g^-$, and $\gamma_g^+$ are positive parameters to be chosen\revblue{, and $|e|$ denotes the diameter of $e$, i.e. the size of the edge in two dimensions.} 
\revblue{ Note that although we consider the case where $\Gamma\cap \partial \Omega = \emptyset$, the method is well defined when the interface crosses the boundary of the domain. However, we only analyze the embedded case. If the interface crosses the boundary of the domain the solution will not necessarily be in $H^2(\Omega^+\cup \Omega^{-})$. In addition,  a technical tool that we utilized in the proof of our main result is the existence of a stable extension (see Lemma 4.1). It is not clear that this extension will exist in some cases where the interface crosses the boundary of the domain.}
 
\begin{remark} We point out that the terms in \eqref{a_h} involving
 integration on $\Gamma$ can be generalized to
\begin{equation*}
\int_{\Gamma} \left(
 \{\rho \nabla v\}_w\,  [u_h] + \{\rho \nabla u_h\}_w \,[v]\right) ds+\sum_{T\in\ThG}\frac{\gamma}{h_T} \tilde{\rho}\int_{T_\Gamma} [u_h][v] ds.
\end{equation*}
The weighted average
\begin{equation*}
\{\rho \nabla v \cdot  \n{n}^-\}_w := (w_{-} \,\rho^- \nabla v_{-} + w_{+} \,\rho^+ \nabla v_{+})  \cdot \n{n}^-\revblue{,}
\end{equation*}
where the weights $w_{-}(x) \in [0,1]$ and $w_{+}(x)= 1 - w_{-}(x)$ and
$\tilde{\rho}$ are chosen properly, see \cite{MR2837483}. The case
$w_{-}(x)=1$ and $\tilde{\rho}(x)= \rho^-(x)$ reduces to the one in (\ref{a_h}).
Another choice considered in the literature (see \cite{MR2002258,MR2257119, MR2491426}) is the harmonic average given by
$w_{-} = \frac{\rho^+}{\rho^+ + \rho^-}$ and $\tilde{\rho} =
\frac{2 \rho^+ \rho^-}{\rho^+ + \rho^-}$. In this case we obtain
\begin{equation*}
\{\rho \nabla v \cdot  \n{n}^-\}_w =  \frac{\tilde{\rho}}{2}\,
(\nabla v_{-} +\nabla v_{+})\cdot \n{n}^-.
\end{equation*}
In this paper we concentrate in the analysis of the choice \eqref{a_h},
however, since $\rho^-\leq \tilde{\rho} \leq 2 \rho^-$, the analysis for
the harmonic average case follows straightforwardly.
\end{remark}

Let us define the broken Sobolev spaces
\begin{equation*}
H_h^2(\Omega_h^\pm)\,\, =\,\,\{v \in H^1(\Omega_h^\pm): v|_{T^\pm} \in H^2(T^\pm), \text{ for all } T \in \mathcal{T}_h^{\pm}  \}.
\end{equation*}
The energy norm $\|\cdot\|_V$, induced by the bilinear form $a_h$, is defined for $v=(v^{-},v^{+}) \in H_h^2(\Omega_h^-) \times H_h^2(\Omega_h^+)$ by
\begin{alignat*}{1}
\|v\|_{V}^2 =&\|\sqrt{\rho} \nabla v\|_{L^2(\Omega)}^2+   \sum_{T\in\ThG}\frac{1}{h_T} \|\sqrt{\rho^-} [v]\|_{L^2(T_\Gamma)}^2 \\
&+\sum_{e \in \EhGm} |e| \,\|\sqrt{\rho^-} \llbracket \nabla v^- \rrbracket \|_{L^2(e)}^2+ \sum_{e \in \EhGp} |e|\, \|\sqrt{\rho^+} \llbracket \nabla v^+ \rrbracket \|_{L^2(e)}^2.
\end{alignat*}
Note that in this definition we use the following notation:
\begin{equation*}
\|\sqrt{\rho} \nabla v\|_{L^2(\Omega)}^2 = \|\sqrt{\rho^{-}} \nabla v^{-}\|_{L^2(\Omega^{-})}^2+\|\sqrt{\rho^{+}} \nabla v^{+}\|_{L^2(\Omega^{+})}^2.
\end{equation*}

\section{Standard a priori error analysis}\label{section3}

\subsection{Stability and best approximation results}
We will need the following technical proposition for the proof of coercivity; proof can be found in Appendix \ref{proofpropostiongeom}.
\begin{proposition}\label{propositiongeom}
Consider a node $z$ of the triangulation $\mathcal{T}_h$ such that $z\in\overline{\Omega}^{-}$. Let $\Delta_z$ be the patch of elements associated to $z$, i.e. $\Delta_z = \mathrm{Int}(\cup\{\overline{T}: T\in\mathcal{T}_h \mbox{ and } z\in \partial T\})$. Then for $h$ small enough, there exists an element $T_{z}\in\Delta_{z}$ such that:
\begin{equation}\label{boundclaim}
|T_{z}\cap\Omega^{-}| \geq C h_{T_{z}}^{2},
\end{equation}
where $C>0$ is a constant independent of $h_{T_{z}}$.
\end{proposition}

Coercivity of the bilinear form $a_h$ is proved below.
\begin{lemma}\label{lemma:coercive}
There exists a constant $c>0$ such that
\begin{equation}\label{coercive}
c\|v\|_V^2 \le  a_h(v,v), \quad \text{ for all } v\in V_h.
\end{equation}
\end{lemma}
\begin{proof}
Let $v \in V_h$. Observe that the bilinear form $a_h$ is symmetric, then it follows
\begin{alignat*}{1}
a_h(v, v)=& \|\sqrt{\rho} \nabla v\|_{L^2(\Omega)}^2 + 2\int_{\Gamma} \rho^- \nabla v^- \cdot \n{n}^-  [v]ds
+  \sum_{T\in\ThG}\frac{\gamma}{h_T}\|\sqrt{\rho^-}[v]\|_{L^2(T_\Gamma)}^2 \\
& + \gamma_g^-\sum_{e \in \EhGm} |e| \rho^- \|\llbracket \nabla v^- \rrbracket \|_{L^2(e^-)}^2+ \gamma_g^+\sum_{e \in \EhGp} |e| \rho^+ \|\llbracket \nabla v^+ \rrbracket \|_{L^2(e)}^2).
\end{alignat*}

In order to prove \eqref{coercive} it is enough to bound  the non positive term (second term). Let $T\in \ThG$. Applying Cauchy-Schwarz inequality we obtain
\begin{equation*}
\left|\int_{\TG} \rho^- \nabla v^- \cdot \n{n}^-  [v] \, ds\right|\le \left( \sqrt{\rho^-h_T}\| \nabla v^- \cdot \n{n}^-\|_{L^2(\TG)} \right)\left(\sqrt{\frac{\rho^-}{h_T}}\| [v]\|_{L^2(\TG)}\right).
\end{equation*}
Summing over $T\in \mathcal{T}_h^{\Gamma}$ and applying arithmetic-geometric inequality give
\begin{equation*}
\left|\int_{\Gamma} \rho^- \nabla v^- \cdot \n{n}^-  [v]ds\right| \le \sum_{T \in \mathcal{T}_h^\Gamma} \left( \varepsilon \rho^-h_T \| \nabla v^- \cdot \n{n}^-\|_{L^2(\TG)} ^2+\frac{\rho^-}{\varepsilon h_T} \|[v]\|_{L^2(\TG)}^2\right).
 \end{equation*}

Let $T \in \mathcal{T}_h^\Gamma$ and let $z\in\overline{\Omega}^{-}$ be a node of $T$. By Proposition \ref{propositiongeom}, there exists a triangle $T_{z}$ satisfying \eqref{boundclaim}. Now, consider the shortest sequence of edges $E(T)=\{e_1,e_2,...,e_N\}$ such that
\begin{equation*}
\left\{
  \begin{array}{ll}
    z\in \overline{e_j}, & j=1,...,N, \\
    e_1\subset \partial T\mbox{ and } e_{N}\subset\partial T_{z},&\\
    e_{j}, e_{j+1} \subset \partial T_j, \,T_j \in\mathcal{T}_h, & j=1,...,N-1.
  \end{array}
\right.
\end{equation*}
Note that by its definition $E(T)\subset \EhGm$. Then, observing that the tangential jump of $\nabla v^-$ is zero along edges, we have
\begin{align*}
h_T \rho^-\| \nabla v^- \cdot \n{n}^-\|_{L^2(\TG)}^2 & \le  h_T\rho^-\frac{|T_{\Gamma}|}{|e_1|} \| \nabla v^-\|^2_{L^2(e_1)} \le \revblue{\kappa^{2}} \rho^- |e_1|\,  \| \nabla v^-\|^2_{L^2(e_1)}\\
                                                     & \le \revblue{\kappa^{2}} \rho^- |e_1| \left( \| \llbracket \nabla v^-\rrbracket\|^2_{L^2(e_1)} + \| \nabla v^-|_{T_1}\|_{L^2(e_1)}^2\right)\\
&\le \revblue{\kappa^{2}} \rho^- |e_1| \,\| \llbracket \nabla v^-\rrbracket\|^2_{L^2(e_1)} + \revblue{\kappa^{4}} \rho^- |e_2|\,\| \nabla v^-|_{T_1}\|_{L^2(e_2)} \\
& \vdots\\
& \leq c(\kappa) \rho^{-}\sum_{e\in E(T)}|e| \,\|\llbracket \nabla v^-\rrbracket\|^2_{L^2(e)} + C(\kappa) \rho^{-} |e_N|\, \| \nabla v^-|_{T_{z}}\|_{L^2(e_N)} \\
& \leq  c(\kappa) \rho^{-}\sum_{e\in E(T)}|e| \,\|\llbracket \nabla v^-\rrbracket\|^2_{L^2(e)}  + \tilde{C}(\kappa) \|\sqrt{\rho^{-}} \nabla v^-\|^2_{L^2(T_{z})},
\end{align*}
where $\kappa$ is the shape regularity constant defined in \eqref{shaperegularity}.\revblue{ Observe that by property \eqref{boundclaim} in Proposition \ref{propositiongeom} and since the test function is piecewise linear, we can estimate the last term above as
\[
\|\sqrt{\rho^{-}} \nabla v^-\|^2_{L^2(T_{z})} \leq C \|\sqrt{\rho^{-}} \nabla v^-\|^2_{L^2(T_{z}\cap \Omega^{-})}.
\]}
\begin{figure}
\begin{center}
\includegraphics[scale=.38]{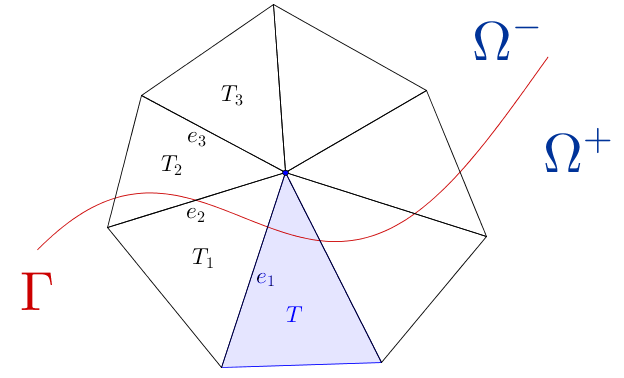}\\
\end{center}
\caption{Illustration of set E(T).}
\label{plotbasis}
\end{figure}
Hence, considering that each nodal patch contains finite number of elements implies that, after summing over $T\in \mathcal{T}_{h}^{\Gamma}$, the terms above are repeated at most finitely many times. Then, it follows
\begin{equation*}
\sum_{T \in \mathcal{T}_h^\Gamma} \rho^-h_T \| \nabla v^- \cdot \n{n}^-\|_{L^2(\TG)} ^2 \le \tilde{C}(\kappa)\left( \|\sqrt{\rho^-} \nabla v\|^{2}_{L^2(\Omega^-)}+\sum_{e \in \EhGm} |e| \rho^- \|\llbracket \nabla v^- \rrbracket \|_{L^2(e)}^2\right).
\end{equation*}
Therefore, coercivity follows by choosing $\gamma$ large enough in terms of $C(\kappa)$  and $\varepsilon$.
\end{proof}

With the aim of proving continuity of the bilinear form $a_h$, we define the following augmented norm:
\begin{equation} \label{augmentednorm}
\|v\|_{V_{A}}^2\,\,=\,\,\|v\|_V^2+  \sum_{T\in\mathcal{T}_h^{\Gamma}}h_T \| \sqrt{\rho^-} \nabla v^- \cdot \n{n}^-\|_{L^2(T_\Gamma)}^2.
\end{equation}

Continuity of the bilinear form follows from its definition and Cauchy-Schwarz inequality. The result is stated as follows:
\begin{lemma} (Continuity) \label{lemma:continuity}
Suppose that $v^\pm,w^\pm \in H_h^2(\Omega^\pm)$. Then, there exists a constant $C>0$, independent of $v$ and $w$, such that
\begin{equation*}
a_h(w,v)\,\, \le\,\,  C \, \|w\|_{V_{A}}  \,  \|v\|_{V_{A}}.
\end{equation*}
Additionally, if $w^\pm \in H_h^2(\Omega^\pm)$ and $v \in V_h$ we have
\begin{equation}\label{continuity}
a_h(w,v) \,\,\le\,\,  C \, \|w\|_{V_{A}}  \,  \|v\|_V.
\end{equation}
\end{lemma}

In order to discuss the Galerkin orthogonality of method \eqref{fem}-\eqref{a_h} we need to define extension operators.
Consider the following well known extension result (see \cite{MR1814364}).
\begin{lemma}\label{Continuousextension}
Assume that $u^{\pm}:=u|_{\Omega^\pm} \in H^{2}(\Omega^{\pm})$. Then, there exist extensions $u_E^{\pm}\in H^{2}(\Omega)$, such that $u_E^{\pm}|_{\Omega^{\pm}} = u^{\pm}$ and
\begin{equation*}
\|u_{E}^{\pm}\|_{H^{i}(\Omega)}\leq C \|u^{\pm}\|_{H^{i}(\Omega^{\pm})}, \quad \text{ for } i=0,1,2,
\end{equation*}
for a constant $C>0$ depending only on $\Omega^{\pm}$.
\end{lemma}
A particular case is proved in Appendix \ref{proofofevenextension} by means of an even extension. From now on we simply  denote $u_{E}^{\pm}$ by $u^\pm$.
\begin{lemma}(Galerkin orthogonality)
Suppose that $u$ solves the \eqref{Problem} and suppose that $u |_{\Omega^\pm} \in H^2(\Omega^\pm)$. Then, we have that
\begin{equation*}
a_h(u,v)=(f,v), \quad \forall v \in V_h,
\end{equation*}
where we use the notation $u=(u^+, u^-)$. Hence, Galerkin orthogonality holds:
\begin{equation*}
a_h(u-u_h,v)=0, \quad \forall v \in V_h.
\end{equation*}
\end{lemma}

The main result of this section, the best approximation result, is stated below. It follows easily from coercivity, continuity and Galerkin orthogonality.
\begin{theorem}\label{bestapproximationthm}(Best approximation)
Let $\Omega\subseteq \mathbb{R}^2$ be an open polygonal domain. Let $u$ be a solution of problem \eqref{Problem} and assume that $u |_{\Omega^\pm} \in H^2(\Omega^\pm)$. Let $u_h$ be solution of the discrete problem \eqref{fem}. Then, there exists a constant $C>0$ independent of $h$, such that
\begin{equation*}\label{bestapproximation}
\|u-u_h\|_{V}\leq C\inf_{v\in V_h} \|u-v\|_{V_{A}}.
\end{equation*}
\end{theorem}

\subsection{Energy error estimates}
In order to prove an error estimate in terms of the $V$-norm we introduce the following interpolation operator:  Define $I_h: H^1(\Omega_h^+)\times H^1(\Omega_h^-)\rightarrow V_h$ such that
\begin{equation*}
(I_h u)^{\pm} = J_h u^{\pm},\quad \mbox{ for } u=(u^+,u^-) \in H^1(\Omega_h^+)\times H^1(\Omega_h^-),
\end{equation*}
where $J_h$ is the interpolant onto the standard continuous piecewise linear polynomials introduced in \cite{MR1011446}. Consequently, the following estimate for the interpolation error follows from the properties of the Scott-Zhang interpolation operator and the  extension result: Lemma \ref{Continuousextension}.
\begin{lemma}\label{interpolationerrorlmm}
Consider the definition of the interpolation operator $I_h$ given above. Then, there exists a constant $C>0$, independent of $h$, such that:
\begin{equation}\label{lemma1}
\|u-I_h u\|_{V_{A}} \le C h (\sqrt{\rho^+} \|D^2 u\|_{L^2(\Omega^+)}+\sqrt{\rho^-} \|D^2 u\|_{L^2(\Omega^-)}).
\end{equation}
\end{lemma}

\revblue{
In addition, we assume that the following elliptic regularity bound holds,
\begin{equation}\label{prop1eqn}
\rho^{+}\| D^{2} u \|_{L^{2}(\Omega^{+})} + \rho^{-}\|D^{2}u\|_{L^{2}(\Omega^{-})} \leq C \|f\|_{L^{2}(\Omega)}.
\end{equation}
For instance, this bound is satisfied in the two-dimensional case (see \cite{MR2684351}) for $\Omega$ convex and polyhedral.}

An energy error estimate follows from Theorem \ref{bestapproximationthm}, Lemma \ref{interpolationerrorlmm} and elliptic regularity \eqref{prop1eqn}. We state it in the corollary below.

\begin{corollary}\label{cor1}(Standard energy error estimate)
Let $u$ be a solution of problem \eqref{Problem}, and let $u_h$ be the solution of the discrete problem \eqref{fem}. Suppose that $\Omega$ is convex. Then, there exists $C>0$ independent of $\rho^\pm$, such that
\begin{equation*}
\|u-u_h\|_{V} \le \frac{C \, h}{\sqrt{\rho^-}} \|f\|_{L^2(\Omega)} .
\end{equation*}
\end{corollary}

\section{Error estimate for the  flux $\rho\nabla(u-u_h)$}\label{section4}
In order to prove the main result of this paper, the error estimate for the flux, we need a discrete extension result.

\begin{lemma}\label{Discreteextensionlmm}(Discrete extension)
Assume that the triangulation $\mathcal{T}_h$ is quasi-uniform. Let $v_h\in V_h^{+}$. Then, there exists a function $E v_h \in V_h^{c}=\{v \in C_0(\Omega): v|_T \in \mathbb{P}^1(T), \forall T \in \mathcal{T}_h\}$, such that $Ev_h = v_h$ in $\Omega_h^{+}$ and
\begin{equation*}
\|Ev_h\|_{H^{1}(\Omega)}\leq C \|v_h\|_{H^1(\Omega_h^{+})},
\end{equation*}
with $C>0$ independent of $h$.
\end{lemma}
\begin{proof}
See Appendix \ref{proofdiscreteextension}.
\end{proof}

Considering the discrete extension lemma, we state a bound for the $H^1$-norm in $\Omega_h^{+}$. The proof follows easily using Proposition \ref{propositiongeom} (with roles of $\Omega^+$ and $\Omega^-$ reversed), \revblue{which allows us to control the terms in $\Omega_{h}^{+}$ by terms in $\Omega^{+}$}. Similar statements have appeared before in \cite{MR2738930} and \cite{MR3268662}.
\begin{lemma}\label{boundextensionlmm}
Let $v\in V_{h}^{+}$. Then, we have
\begin{equation*}
\|v\|^{2}_{H^{1}(\Omega_{h}^{+})}\leq C\left(\|v\|^{2}_{H^1(\Omega^{+})}+\sum_{e\in\EhGp}|e|\,\|\llbracket \nabla v\rrbracket\|^{2}_{L^{2}(e)}\right).
\end{equation*}
\end{lemma}
\revblue{
\begin{proof}
We first bound the $H^1$ semi-norm. By the argument given in proof of Lemma \ref{lemma:coercive} which uses Proposition \ref{propositiongeom} it follows that
\[
\| \nabla v\|^{2}_{L^{2}(\Omega_{h}^+)}\leq C\left( \|\nabla v\|_{L^{2}(\Omega^{+})}^{2} + \sum_{e\in\EhGp}|e|\,\|\llbracket \nabla v\rrbracket\|^{2}_{L^{2}(e)} \right)
\]
where $C$ depends on the shape regularity constant $\kappa$. Now, in order to bound the $L^2$ norm let $T$ be an element $\Omega_h^+$ not totally contained in $\Omega^+$. Let $z$ be a node of $T$ such that $z\in \Omega^{+}$. Then, by Proposition \ref{propositiongeom}, there exists an element $T_{z}$ in the patch of $z$ such that :
\[ |T_{z}\cap \Omega^{+}| \geq C h_{T_{z}}^{2}. \]
Then, using the fact that $v$ is a linear function on $T$ we have that
\[
|v(x)| \leq |v(z)| + h_{T}|\nabla v|_{T}|,\quad \mbox{for all } x \in T,
\]  
and then we have
\[
 \|v\|_{L^{2}(T)} \leq C h_{T} |v(z)| + h_{T} \|\nabla v\|_{L^{2}(T)}. 
 \] 
Using an inverse estimate we get
\[
|v(z)| \leq C h^{-1}_{T_{z}} \|v\|_{L^{2}(T_{z})} .
\]
It is not difficult to see, using the fact that $|T_{z}\cap \Omega^{+}| \geq C h_{T_{z}}^{2}$  and that $v$ is linear on $T_z$ that 
\[
\|v\|_{L^{2}(T_{z})}  \le C\, \|v\|_{L^{2}(T_{z}\cap \Omega^{+})}.
\]
Therefore,  we obtain 
\[\|v\|_{L^{2}(T)} \leq C \|v\|_{L^{2}(T_{z}\cap \Omega^{+})} +  h_{T} \|\nabla v\|_{L^{2}(T)}.\]
Using this inequality repeatedly we obtain that
\[ \|v\|_{L^{2}(\Omega^+_{h})} \leq C\left(\|v\|_{L^{2}(\Omega^+)} + h \|\nabla v\|_{L^{2}(\Omega_h^+)}\right),\]
which proves the result after applying the estimate derived for the $H^1$ semi-norm.
\end{proof}
}

Now we are in position to state and prove the main result of this paper. \revblue{Observe that if we have a  conforming piecewise linear discretization with an interface $\Gamma$ aligning the mesh $\mathcal{T}_h$, then the proof would be short. Precisely, the fact that the mesh does not align the interface creates some extra technicalities on the proof.} 
\begin{theorem}\label{Maintheorem}(Main result)
Let $u$ be a solution of problem \eqref{Problem} and let $u_h$ be solution of the discrete problem \eqref{fem}. Assume the triangulation is quasi-uniform and that   \revblue{$(\rho^+ \| D^2 u\|_{L^2(\Omega^+)}+ \rho^- \|D^2 u\|_{L^2(\Omega^-)})$ is bounded} .Then, there exists a constant $C>0$, independent of $h$ and $\rho^{\pm}$, such that
\begin{equation*}
\|\rho\nabla(u-u_h)\|_{L^{2}(\Omega)}\leq C h   \revblue{(\rho^+ \| D^2 u\|_{L^2(\Omega^+)}+ \rho^- \|D^2 u\|_{L^2(\Omega^-)} )}.  
\end{equation*}
We remind the reader that
\begin{equation*}
\|\rho\nabla(u-u_h)\|_{L^{2}(\Omega)}^2=\|\rho^{-}\nabla(u-u_h)^-\|_{L^{2}(\Omega^-)}^2+\|\rho^{+}\nabla(u-u_h)^+\|_{L^{2}(\Omega^+)}^2.
\end{equation*}
\end{theorem}
\begin{proof}
Observe that the bound in $\Omega^{-}$ is given by \revblue{ Theorem \ref{bestapproximationthm} and \eqref{lemma1}} 
\begin{equation*}
\|\rho^- \nabla(u-u_h)^{-}\|_{L^2(\Omega^-)} \le C h   \revblue{(\rho^+ \| D^2 u\|_{L^2(\Omega^+)}+ \rho^- \|D^2 u\|_{L^2(\Omega^-)} )}. 
\end{equation*}
Thus, it remains to prove
\begin{equation}\label{rhoperror}
\|\rho^+ \nabla(u-u_h)^{+}\|_{L^2(\Omega^+)} \le C\, h   \revblue{(\rho^+ \| D^2 u\|_{L^2(\Omega^+)}+ \rho^- \|D^2 u\|_{L^2(\Omega^-)} )}. 
\end{equation}

Consider the function $v_h = (I_h u - u_h)^{+}\in V_h^{+}$. If $\Omega^+$ is the inclusion, i.e., $\partial \Omega^{+}\cap\partial \Omega = \emptyset$, we redefine $v_h$ so that it has average zero on $\Omega^{+}$, i.e., $v_h = v_h -\frac{1}{|\Omega^{+}|}\int_{\Omega^{+}} v_h $. Either case Poincare's inequality holds:
\begin{equation*}\label{Poincare}
\|v_h\|_{H^{1}(\Omega^{+})}\leq C \|\nabla v_h\|_{L^{2}(\Omega^{+})}=C \|\nabla (I_h u - u_h)^{+}\|_{L^{2}(\Omega^{+})},
\end{equation*}
for a constant $C>0$. Applying Lemma \ref{Discreteextensionlmm} it follows that there exists an extension $E v_h \in V_h^{c}$ such that  $Ev_h = v_h$ in $\Omega_h^{+}$, and $\|Ev_h\|_{H^{1}(\Omega)}\leq C \|v_h\|_{H^1(\Omega_h^{+})}$. Combining this estimate with Lemma \ref{boundextensionlmm}, we have
\begin{equation}\label{boundextension}
\|Ev_h\|_{H^{1}(\Omega)}\leq C\Big( \|\nabla(I_h u - u_h)^{+}\|_{L^2(\Omega^{+})}+\Big(\sum_{e\in\EhGp}|e|\,\|\llbracket \nabla (I_h u - u_h)^{+}\rrbracket\|^{2}_{L^{2}(e)}\Big)^{1/2}\Big).
\end{equation}
Consequently, using $((E v_h)^-, (E v_h)^+)  \in V_h$ as test function (where $(E v_h)^\pm := Ev_h|_{\Omega_h^\pm}$), and the consistency of the method \eqref{fem}-\eqref{a_h}, we obtain
\begin{alignat*}{1}
 0&=a_h(u-u_h, Ev_h)\\
 &=\int_{\Omega^+} \rho^+ \nabla(u-u_h)^{+}\cdot \nabla (Ev_h)^{+}+\int_{\Omega^-} \rho^- \nabla(u-u_h)^{-} \cdot \nabla (Ev_h)^{-}\\
&+ \int_{\Gamma} \rho^- \nabla Ev_h^- \cdot \n{n}^-  [u-u_h]ds \\
&+ \sum_{e \in \EhGm} {|e|} \int_{e} \rho^-  \llbracket \nabla (Ev_h)^{-}\rrbracket \llbracket \nabla (u-u_h)^- \rrbracket \,ds  \\
&+ \sum_{e \in \EhGp} {|e|} \int_{e} \rho^+  \llbracket \nabla (Ev_h)^{+}\rrbracket \llbracket \nabla (u-u_h)^+ \rrbracket \,ds.
\end{alignat*}
Here we used that $E v_h$ is continuous across the interface $\Gamma$. Multiplying by $\rho^{+}$ and adding and subtracting $I_h u^{+}$ in the first and last term give
\begin{alignat*}{1}
&\| \rho^+ \nabla(I_h u-u_h)^{+}\|_{L^2(\Omega^+)}^2+ \sum_{e \in \EhGp} {|e|}\, \| \rho^{+}\llbracket \nabla (I_hu-u_h)^+\rrbracket  \|_{L^2(e)}^2=  \\
&-\rho^+ \int_{\Omega^+} \rho^+ \nabla(u-I_h u)^{+} \cdot \nabla (E v_h)^{+} - \rho^+ \sum_{e \in \EhGp} {|e|} \int_{e} \rho^+ \llbracket \nabla (Ev_h)^{+} \rrbracket \llbracket \nabla (u-I_h u)^+ \rrbracket  \,ds  \\
&-\rho^+ \int_{\Omega^-} \rho^- \nabla(u-u_h)^{-} \cdot \nabla (Ev_h)^{-} -\rho^+ \int_{\Gamma} \rho^- \nabla (Ev_h)^{-} \cdot \n{n}^-  [u-u_h]ds \\
&-\rho^+ \sum_{e \in \EhGm} {|e|} \int_{e} \rho^- \llbracket \nabla   (Ev_h)^{-} \rrbracket  \llbracket \nabla (u-u_h)^-\rrbracket  \,ds  =:I_1+I_2,
\end{alignat*}
where $I_1$ denotes the first two terms and $I_2$ the last three terms. First we bound $I_1$. Applying Cauchy-Schwarz and arithmetic geometric inequalities, and using that $\nabla (Ev_h)^{+} = \nabla (I_h u -u_h)^{+}$ in  $\Omega_h^+$ yield
\begin{alignat*}{1}
2|I_1|\le &    \| \rho^+ \nabla(I_h u-u)^{+}\|_{L^2(\Omega^+)}^2+ \sum_{e \in \EhGp} {|e|} \, \|\rho^{+} \llbracket \nabla  (I_hu-u)^+ \rrbracket   \|_{L^2(e)}^2  \\
&+ \| \rho^+ \nabla(I_h u-u_h)^{+}\|_{L^2(\Omega^+)}^2+ \sum_{e \in \EhGp} {|e|} \, \|\rho^{+} \llbracket \nabla  (I_hu-u_h)^+ \rrbracket   \|_{L^2(e)}^2\\
\le &    \| \rho^+ \nabla(I_h u-u)^{+}\|_{L^2(\Omega^+)}^2+ \sum_{e \in \EhGp} {|e|} \, \|\rho^{+} \llbracket \nabla  (I_hu-u)^+ \rrbracket   \|_{L^2(e)}^2  +|I_1| + |I_2|,
\end{alignat*}
\revblue{ which implies that
\begin{equation}\label{I1I2firstestimate}
|I_1|\leq  \| \rho^+ \nabla(I_h u-u)^{+}\|_{L^2(\Omega^+)}^2+ \sum_{e \in \EhGp} {|e|} \, \| \rho^+ \llbracket \nabla (I_hu-u)^+ \rrbracket  \|_{L^2(e)}^2 + |I_2|. 
\end{equation} }
Similarly, we apply Cauchy-Schwarz inequality for the three terms in $I_2$ after taking as common factor $\rho^{+}\sqrt{\rho^{-}}$, i.e.
\begin{alignat*}{1}
&\left|\sqrt{\rho^{-}} \int_{\Omega^-} \nabla(u-u_h)^{-} \cdot \nabla (Ev_h)^{-} dx \right|\le \|\sqrt{\rho^{-}}\nabla(u-u_h)^{-}\|_{L^{2}(\Omega^{-})}\|\nabla (E v_h)^{-}\|_{L^{2}(\Omega^{-})}, \\
&\left|\sqrt{\rho^{-}} \int_{\Gamma} \nabla (Ev_h)^- \cdot \n{n}^-  [u-u_h]ds\right| \le  \sum_{T\in\ThG}\frac{1}{\sqrt{h_T}} \|\sqrt{\rho^{-}}\left[u-u_h\right]\|_{L^{2}(T_{\Gamma})}\sqrt{h_T}\|\nabla (E v_h)^{-}\cdot \n{n}^{-}\|_{L^{2}(T_{\Gamma})}, \\
&\Big|\sqrt{\rho^{-}} \sum_{e \in \EhGm} |e|\, \int_{e} \llbracket \nabla  (Ev_h)^- \rrbracket  \llbracket \nabla (u-u_h)^-\rrbracket ds \Big|\le \sum_{e \in \EhGm} |e|\|\sqrt{\rho^{-}}\llbracket \nabla(u-u_h)^{-}\rrbracket \|_{L^{2}(e)}\|\llbracket\nabla (E v_h)^{-}\rrbracket \|_{L^{2}(e)}.
\end{alignat*}
Observe that terms involving $u-u_h$ are all bounded by $\|u-u_h\|_V$. Hence,
\begin{alignat*}{1}
|I_2| \le  C\, &\rho^+ \sqrt{\rho^-} \|u-u_h\|_{V} \Big( \| \nabla (E v_h)^{\revblue{-}}\|_{L^2(\Omega^{\revblue{-}})} \\
&+\big(\sum_{T \in \ThG} h_T \|\nabla (E v_h)^{-}\cdot \n{n}^{-}\|_{L^2(\TG)}^2\big)^{1/2}+ \big(\sum_{e \in \EhGm}|e|\, \|\llbracket \nabla (E v_h)^{-} \rrbracket\|_{L^2(e)}^2\big)^{1/2}\Big).
\end{alignat*}
\revblue{For the terms involving $E v_h$ we use the argument in the proof of Proposition \ref{propositiongeom} to bound term in $T_{\Gamma}$ by terms in a region of size $h_{T}^{2}$. In addition, applying inverse inequalities we obtain}
\begin{equation*}
\big(\sum_{T \in \ThG} h_T \|\nabla (E v_h)^{-}\cdot \n{n}^{-}\|_{L^2(\TG)}^2\big)^{1/2}+ \big(\sum_{e \in \EhGm}|e| \,\|\llbracket \nabla (E v_h)^{-} \rrbracket\|_{L^2(e)}^2\big)^{1/2}
\le C \| \nabla E v_h\|_{L^2(\Omega)}.
\end{equation*}
Thus, \revblue{by estimate \eqref{I1I2firstestimate} and the estimate above  for $I_2$, it follows that}
\begin{alignat*}{1}
\revblue{|I_1|+|I_2|}\le &  \| \rho^+ \nabla(I_h u-u)^{+}\|_{L^2(\Omega^+)}^2+ \sum_{e \in \EhGp} {|e|} \, \| \rho^+ \llbracket \nabla (I_hu-u)^+ \rrbracket  \|_{L^2(e)}^2  \\
&+ C\, \sqrt{\rho^-} \|u-u_h\|_V  \|\rho^{+}\nabla E v_h\|_{L^2(\Omega)}.
\end{alignat*}
By approximation properties of the Scott-Zhang interpolant we have
\begin{alignat}{1}
\| \rho^+ \nabla(I_h u-u)^{+}\|_{L^2(\Omega^+)}^2&+ \sum_{e \in \EhGp} {|e|}  \| \rho^+ \llbracket \nabla (I_hu-u)^+ \rrbracket  \|_{L^2(e)}^2 \label{rhoI}\\ &\le   C h^2(\rho^{+})^2 \|D^2 u^+\|^2_{L^2(\Omega^{+})}. \nonumber
\end{alignat}
Using the error estimate   \revblue{ Theorem \ref{bestapproximationthm} and \eqref{lemma1}}  and definitions of $I_1$ and $I_2$ we conclude
\begin{alignat*}{1}
\|\rho^+ \nabla(I_h u-u_h)^{+}\|_{L^2(\Omega^+)}^2+ &\sum_{e \in \EhGp} {|e|}\,  \| \rho^+ \llbracket \nabla (I_hu-u_h)^+ \rrbracket  \|_{L^2(e)}^2  \\
\leq\,&C (h  \|\rho^+ \nabla E v_h\|_{L^2(\Omega)} \revblue{(\rho^+ \| D^2 u\|_{L^2(\Omega^+)}+ \rho^- \|D^2 u\|_{L^2(\Omega^-)} )} \\
       &\quad + h^2 \revblue{ (\rho^{+})^2 \|D^2 u^+\|^2_{L^2(\Omega^{+})}}).
\end{alignat*}
Therefore, applying \eqref{boundextension} we have
\begin{equation*}
\|\rho^+ \nabla(I_h u-u_h)\|_{L^2(\Omega^+)}^2+ \sum_{e \in \EhGp} {|e|} \, \| \rho^+ \llbracket \nabla (I_hu-u_h)^+ \rrbracket  \|_{L^2(e)}^2  \le Ch^2  \revblue{(\rho^+ \| D^2 u\|_{L^2(\Omega^+)}+ \rho^- \|D^2 u\|_{L^2(\Omega^-)} )^2}  .
\end{equation*}
Finally, applying triangle inequality, the estimate for the interpolation error \eqref{rhoI}, and last inequality we obtain the desired bound \eqref{rhoperror}.
\end{proof}

\revblue{We can now use the elliptic regularity result to prove the following corollary.} 
\begin{corollary}
\revblue{Assuming the hypothesis of Theorem \ref{Maintheorem} and in addition assuming that our domain is such that  \eqref{prop1eqn} holds, then }
\begin{equation*}
\revblue{ \|\rho\nabla(u-u_h)\|_{L^{2}(\Omega)} \leq C h  \|f\|_{L^2(\Omega)}.}
\end{equation*}
 
\end{corollary}

We conclude this section by stating an $L^2$-error estimate.  The proof follows easily from a duality argument (see \cite{GSS15}). We omit the details.
\begin{lemma}($L^2$ error estimate)
Assuming the hypothesis of \revblue{previous  corollary}  we have
\begin{equation*}
\|(u-u_h)^-\|_{L^2(\Omega^-)}+\|(u-u_h)^+\|_{L^2(\Omega^+)} \le \frac{C}{\rho^{-}} h^2 \|f\|_{L^2(\Omega)}.
\end{equation*}
\end{lemma}

\section{\revblue{Extensions of the method}}\label{extensions}

In this section we discuss and state some straightforward extensions of method \eqref{fem}-\eqref{a_h}.
\subsection{Non homogeneous jump conditions}
We consider problem \eqref{Problem} with a non homogeneous jump conditions in equations \eqref{Problem:c} and \eqref{Problem:d}, i.e.
\begin{subequations}\label{Problemflux}
\begin{alignat}{3}
    - \nabla\cdot(\rho^{\pm} \nabla u^{\pm}) &= f^{\pm} \qquad &\mbox{in } &\Omega^{\pm},  \\
                                           u^{\pm} &=0            &\mbox{on }&\partial \Omega^{\pm}\backslash\Gamma, \\
                            \left[u\right]&=\alpha            &\mbox{on }&\Gamma,  \\
  \left[\rho \nabla  u\cdot \n{n}  \right]&=\beta           &\mbox{on }&\Gamma,
\end{alignat}
\end{subequations}
where $\alpha, \beta $ are a smooth functions given on the interface. The method in this case follows from a standard derivation of the variational formulation:
\begin{alignat}{1}
a_h(u_h, v)&=(f^{+},v^{+})_{\Omega^+}+(f^{-},v^{-})_{\Omega^-}+ \int_{\Gamma} \big(\beta v^{+} +\rho^{-}\nabla v^{-}\cdot\n{n}^{-} \alpha\big)ds\label{femnonzeroflux} \\& +\sum_{T\in\ThG} \frac{\gamma}{h_T}\rho^{-} \int_{T_\Gamma} [v]\alpha ds \nonumber,
\end{alignat}
for all $v\in V_h$.

All the proofs would generalize easily. \revblue{In particular, the main result Theorem \ref{Maintheorem} holds}.  \revblue{ However, if one wants a result in terms of the data of the problem (e.g. $f$, $\beta$ and $\alpha$) then one would need a regularity result like the one   in \cite{MR2684351} for this more complicated problem which does not seem to appear in the literature.}

\subsection{Three dimensional problem}
To extend the method in three dimensions is straightforward. In order to prove the same result one needs the regularity results from \cite{MR2684351} to hold in three dimensions which we have not found in the literature. Moreover, we need to be able to prove the geometric result and extension results in three dimensions: Proposition \ref{propositiongeom} and Lemma \ref{Discreteextensionlmm}. We believe that these results should hold.

\section{Numerical examples}\label{sectionnumerical}
This section illustrates numerically the a priori error estimates proved in Section \ref{section3} and  Section \ref{section4}. We consider a two dimensional example with a non trivial immersed and closed interface. In particular, the example supports the optimal order of convergence for the error of the flux $\rho\nabla(u-u_h)$. We summarize our experimental results in tables, displaying the following errors and experimental orders of convergence (eoc):
\begin{align*}
e_h^0\,\,& :=\,\, \|u -  u_h\|_{L^2(\Omega)},&  e_h^\infty\,\,& :=\,\, \|u - u_h\|_{L^\infty(\Omega)},   \\
e_{h,\rho}^1\,\,& :=\,\, \|\rho(\nabla u - \nabla u_h)\|_{L^2(\Omega)},&
e_{h,\rho}^{1,\infty}\,\,& :=\,\, \|\rho(\nabla u -  \nabla u_h)\|_{L^\infty(\Omega)} ,
\end{align*}
\[ \mathrm{eoc}(e) \,\,:=\,\, \frac{\log(e_{h_{l+1}}/e_{h_l})}{ \log(h_{l+1}/h_{l})}.\]
Our theoretical results predict optimal convergence of: error $e_h^0$ (second order), and error $e_h^1$ (first order).  We also test the convergence of the errors $ e_h^\infty$ and $e_h^{1,\infty}$.

The finite element approximation by scheme \eqref{fem}-\eqref{a_h} is computed with a sequence of uniform triangulations non matching the interface. The parameter of the triangulation is given by $h = 2^{-(l+3/2)}$, for $l = 1,\,...,\,7$. Computations were performed in MATLAB including the solution of the linear system by means of command "$\backslash$".

\begin{enumerate}[label=\bfseries (\arabic*)]
\item\label{Ex1} Consider problem \eqref{Problem} in a square domain $\Omega = (-1,1)^2$ and an immersed interface $\Gamma = \{ x \in\Omega : x_1^2+x_2^2 = (1/3)^2\}$. We  will test both cases: $\Omega^{-}$ as inclusion, and $\Omega^{+}$ as inclusion. Consider the following exact solution:
\begin{equation}\label{exactsolution1}
    u(x)\,\,=\,\, \left\{
                        \begin{array}{ll}
                          \frac{r^\alpha}{\rho^-} &, \hbox{ if } x\in \Omega^-, \\
                          \frac{r^\alpha}{\rho^+} +(1/3)^\alpha(\frac{1}{\rho^-} - \frac{1}{\rho^+}) &,  \hbox{ if } x\in \Omega^+, \
                        \end{array}
                      \right.
\end{equation}
where $r = \sqrt{x_1^2+x_2^2}$ and $\alpha = 2$. We test with values of the diffusion coefficients $\rho^+ = 10^4$ and $\rho^-=1$.  Tables \ref{Table:Ex1_1} and  \ref{Table:Ex1_2} summarize the results obtained by method \eqref{fem}-\eqref{a_h} with stabilization parameters $\gamma = 10$ and $\gamma_g^\pm=10$, with $\Omega^{-}$ and $\Omega^{+}$  as inclusion, respectively.

\begin{table}[!htp]
\footnotesize
\ra{1.1}
\begin{center}
\begin{tabular}{@{}l@{\hskip .3in}c@{\hskip .2in}c@{\hskip .4in}c@{\hskip .2in}c@{\hskip .4in}c@{\hskip .2in}c@{\hskip .4in}c@{\hskip .2in}c@{}}\toprule
$l$ & $e_h^0$ & eoc & $e_h^\infty$ & eoc  & $e_{h,\rho}^1$ & eoc & $e_{h,\rho}^{1,\infty}$ & eoc \\ \midrule
1   &  1.9e-02  &          &  5.5e-02  &          &   3.7e-01  &         &  5.9e-01 &       \\
2   &  6.2e-03  &   1.64  &  1.5e-02  &   1.86  &   1.4e-01  &  1.38  &  2.7e-01  &  1.11 \\
3   &  1.1e-03  &   2.46  &  2.8e-03  &   2.45  &   5.6e-02  &  1.31  &  1.2e-01  &  1.17 \\
4   &  1.7e-04  &   2.69  &  5.0e-04  &   2.48  &   2.6e-02  &  1.11  &  5.4e-02  &  1.17 \\
5   &  2.8e-05  &   2.66  &  9.8e-05  &   2.34  &   1.3e-02  &  1.03  &  2.4e-02  &  1.14 \\
6   &  4.6e-06  &   2.59  &  1.9e-05  &   2.39  &   6.4e-03  &  1.01  &  1.2e-02  &  1.03 \\
7   &  8.4e-07  &   2.45  &  4.2e-06  &   2.16  &   3.2e-03  &  1.00  &  6.0e-03  &  1.01 \\
 \bottomrule
  \end{tabular}
\end{center}
 \vskip3mm

\caption{Example \ref{Ex1}: errors and experimental orders of convergence (eoc) for problem \eqref{Problem} and solution \eqref{exactsolution1} with $\Omega^{-} = \{ x\in\Omega: x_1^2+x_2^2<(1/3)^2\}$ and $\Omega^{+} = \Omega\backslash (\Omega^{-}\cup\Gamma)$.}
\label{Table:Ex1_1}
\end{table}
\begin{table}[!htp]
\footnotesize
\ra{1.1}
\begin{center}
\begin{tabular}{@{}l@{\hskip .3in}c@{\hskip .2in}c@{\hskip .4in}c@{\hskip .2in}c@{\hskip .4in}c@{\hskip .2in}c@{\hskip .4in}c@{\hskip .2in}c@{}}\toprule
$l$ & $e_h^0$ & eoc & $e_h^\infty$ & eoc  & $e_{h,\rho}^1$ & eoc & $e_{h,\rho}^{1,\infty}$ & eoc \\ \midrule
1   &  3.7e-02  &         &  5.9e-02  &         &   3.6e-01  &        &  6.0e-01 &       \\
2   &  8.3e-03  &   2.15  &  1.3e-02  &   2.23  &   1.4e-01  &  1.41  &  2.7e-01  &  1.15 \\
3   &  1.5e-03  &   2.45  &  2.9e-03  &   2.09  &   5.6e-02  &  1.29  &  1.2e-01  &  1.17 \\
4   &  2.7e-04  &   2.50  &  5.8e-04  &   2.35  &   2.6e-02  &  1.09  &  5.3e-02  &  1.18 \\
5   &  5.2e-05  &   2.35  &  9.8e-05  &   2.56  &   1.3e-02  &  1.02  &  2.4e-02  &  1.14 \\
6   &  1.2e-05  &   2.16  &  2.1e-05  &   2.25  &   6.4e-03  &  1.01  &  1.2e-02  &  1.01 \\
7   &  2.8e-06  &   2.06  &  4.2e-06  &   2.29  &   3.2e-03  &  1.00  &  5.8e-03  &  1.05 \\
 \bottomrule
  \end{tabular}
\end{center}
 \vskip3mm

\caption{Example \ref{Ex1}: errors and experimental orders of convergence (eoc)  for problem \eqref{Problem} and solution \eqref{exactsolution1} with $\Omega^{+} = \{ x\in\Omega: x_1^2+x_2^2<(1/3)^2\}$ and $\Omega^{-} = \Omega\backslash (\Omega^{+}\cup\Gamma)$.}
\label{Table:Ex1_2}
\end{table}
Optimal convergence of the errors, second order for $e_h^0$ and $e_h^\infty$ and first order for $e_{h,\rho}^1$ and $e_{h,\rho}^\infty$, for Example \ref{Ex1} is observed in Tables \ref{Table:Ex1_1} and \ref{Table:Ex1_2} for both cases $\Omega^{-}$ and $\Omega^{+}$ as inclusion.

In addition, in order to test the independence of the coefficient $\rho^{\pm}$ of our estimate in Theorem \ref{Maintheorem}, we consider the exact solution \eqref{exactsolution1} with a fix mesh corresponding to $h = 2^{-(5+3/2)}$. We compute errors $e_h^0$, $e_{h,\rho}^1$ and $e_{h,\sqrt{\rho}}^1 = \|\sqrt{\rho}(u-u_h)\|_{L^2(\Omega)}$ for decreasing values of $\rho^{-}$ and increasing values of $\rho^+$. We summarize the results in Tables \ref{Table:Ex1_3} and \ref{Table:Ex1_4}, corresponding to $\Omega^{-}$ as inclusion and $\Omega^{+}$ as inclusion, respectively.

\begin{table}[!htp]
\footnotesize
\ra{1.1}
\begin{center}
\begin{tabular}{@{}l@{\hskip .2in}l@{\hskip .4in}c@{\hskip .4in}c@{\hskip .4in}c@{}}\toprule
$\rho^{-}$ & $\rho^{+}$ &   $e_h^0$    &   $e_{h,\rho}^1$  &$e_{h,\sqrt{\rho}}^1$\\ \midrule
1e+00      &  1e+01     &   3.0e-05    &   1.3e-02  &     5.5e-03 \\
1e-01      &  1e+02     &   2.8e-04    &   1.3e-02  &     1.3e-02 \\
1e-02      &  1e+03     &   2.8e-03    &   1.3e-02  &     4.0e-02 \\
1e-03      &  1e+04     &   2.8e-02    &   1.3e-02  &     1.3e-01 \\
1e-04      &  1e+05     &   2.8e-01    &   1.3e-02  &     4.0e-01 \\
 \bottomrule
  \end{tabular}
\end{center}
 \vskip3mm
\caption{Example \ref{Ex1}: errors for mesh parameter $h = 2^{-(5+3/2)}$. $\Omega^{-} = \{ x\in\Omega: x_1^2+x_2^2<(1/3)^2\}$ and $\Omega^{+} = \Omega\backslash (\Omega^{-}\cup\Gamma)$.}
\label{Table:Ex1_3}
\end{table}
\begin{table}[!htp]

\footnotesize
\ra{1.1}
\begin{center}
\begin{tabular}{@{}l@{\hskip .2in}l@{\hskip .4in}c@{\hskip .4in}c@{\hskip .4in}c@{}}\toprule
$\rho^{-}$ & $\rho^{+}$ & $e_h^0$    &    $e_{h,\rho}^1$   &   $e_{h,\sqrt{\rho}}^1$   \\ \midrule
1e+00      &  1e+01      &   5.4e-05  &     1.3e-02         &     1.2e-02      \\
1e-01      &  1e+02      &   5.2e-04  &     1.3e-02         &     3.9e-02      \\
1e-02      &  1e+03      &   5.2e-03  &     1.3e-02         &     1.2e-01      \\
1e-03      &  1e+04      &   5.3e-02  &     1.3e-02         &     3.9e-01      \\
1e-04      &  1e+05      &   3.9e-01  &     1.3e-02         &     1.2e+00      \\
 \bottomrule
  \end{tabular}
\end{center}
 \vskip3mm

\caption{Example \ref{Ex1}: errors for mesh parameter $h = 2^{-(5+3/2)}$. $\Omega^{+} = \{ x\in\Omega: x_1^2+x_2^2<(1/3)^2\}$ and $\Omega^{-} = \Omega\backslash (\Omega^{+}\cup\Gamma)$.}
\label{Table:Ex1_4}
\end{table}

Tables \ref{Table:Ex1_3} and \ref{Table:Ex1_4} show that error $e_{h,\rho}^1$ is practically invariant, corroborating that in the main result of our paper Theorem \ref{Maintheorem}, the estimate is totally independent of the diffusion coefficients $\rho^{\pm}$. Errors $e_h^0$ and $e_{h,\sqrt{\rho}}^1$ seems to be dependent of the coefficients as our estimates in Section \ref{section3} show.
\revblue{
\item\label{ExampleNitschesflower}
Consider the two dimensional domain $\Omega = (-1,1)^2$ with the immersed interface $\Gamma$ defined by $\Gamma = \{x = (x_1,x_2)\in \Omega: \|x\|_2 = r = 1/18+0.2\sin(5s),\,s\in[0,2\pi)\}$. We define $\Omega^{-}$ as the interior domain, i.e. $\partial \Omega^{-} = \Gamma$. We further consider the following exact solution
    \begin{equation}\label{exactsolution3}
    u(x)= \left\{
                        \begin{array}{ll}
                          \frac{1}{\rho^-}(x_1^2+x_2^2)^2 &, \hbox{ if } (x,y)\in \Omega^-, \\
                          \frac{1}{\rho^+}x_2\sqrt{x_1^2+x_2^2} &,  \hbox{ if } (x,y)\in \Omega^+, \
                        \end{array}
                      \right.
\end{equation}
We set the stabilization parameters to be: $\gamma = 10$, $\gamma_g^+=10$ and $\gamma_g^-=10$.
Note that in this case the jump of the solution and the jump of the flux are nonzero. Table \ref{Table:Ex3_1} shows the errors and experimental orders of convergence obtained by method \eqref{femnonzeroflux}-\eqref{a_h}.  As in the previous example, we test the independence of the coefficient $\rho^{\pm}$ of our estimate in Theorem \ref{Maintheorem}. We consider the exact solution \eqref{exactsolution3} with a fix mesh corresponding to $h = 2^{-(5+3/2)}$. We compute errors $e_h^0$, $e_{h,\rho}^1$ and $e_{h,\sqrt{\rho}}^1 = \|\sqrt{\rho}(u-u_h)\|_{L^2(\Omega)}$ for decreasing values of $\rho^{-}$ and increasing values of $\rho^+$. We summarize these results in Table \ref{Table:Ex3_2}. }

\revblue{
We observe optimal convergence of the errors, second order for $e_h^0$ and $e_h^\infty$ and first order for $e_{h,\rho}^1$ and $e_{h,\rho}^\infty$, for Example \ref{ExampleNitschesflower} in Table \ref{Table:Ex3_1}. Table \ref{Table:Ex3_2} shows that error $e_{h,\rho}^1$ is practically invariant, supporting our claim that the estimate for the error of the flux is totally independent of the diffusion coefficients $\rho^{\pm}$ for the case of non-homogeneous jumps.
\begin{table}[!htp]
\footnotesize
\ra{1.1}
\begin{center}
\begin{tabular}{@{}l@{\hskip .4in}c@{\hskip .2in}c@{\hskip .4in}c@{\hskip .2in}c@{\hskip .4in}c@{\hskip .2in}c@{\hskip .4in}c@{\hskip .2in}c@{}}\toprule
$l$ & $e_h^0$ & eoc & $e_h^\infty$ & eoc  & $e_{h,\rho}^1$ & eoc & $e_{h,\rho}^{1,\infty}$ & eoc \\ \midrule
1   &  2.4e-2  &   $-$   &  7.1e-2  &   $-$   &   5.1e-1  &  $-$   &  1.0e+0 &    $-$ \\
2   &  1.0e-2  &   1.23  &  2.6e-2  &   1.46  &   2.2e-1  &  1.18  &  7.8e-1  &  0.39 \\
3   &  3.2e-3  &   1.68  &  1.4e-2  &   0.91  &   9.8e-2  &  1.19  &  5.1e-1  &  0.62 \\
4   &  8.6e-4  &   1.91  &  4.4e-3  &   1.63  &   3.7e-2  &  1.41  &  2.9e-1  &  0.84 \\
5   &  1.7e-4  &   2.35  &  1.0e-3  &   2.13  &   1.4e-2  &  1.43  &  1.4e-1  &  1.07 \\
6   &  2.8e-5  &   2.60  &  1.9e-4  &   2.42  &   5.9e-3  &  1.23  &  6.2e-2  &  1.14 \\
 \bottomrule
  \end{tabular}
\end{center}
 \vskip3mm

\caption{Example \ref{ExampleNitschesflower}: errors for mesh parameter $h = 2^{-(5+3/2)}$ , using method \eqref{femnonzeroflux}-\eqref{a_h}.}
\label{Table:Ex3_1}
\end{table}
\begin{figure}[!htp]
\begin{center}
  \includegraphics[height = 5cm, width=6cm]{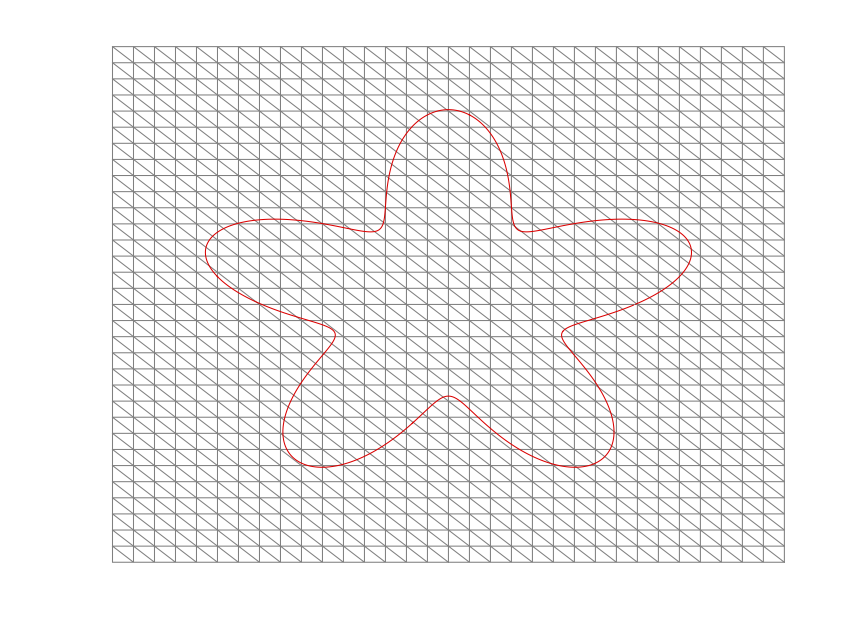}\includegraphics[scale=.5]{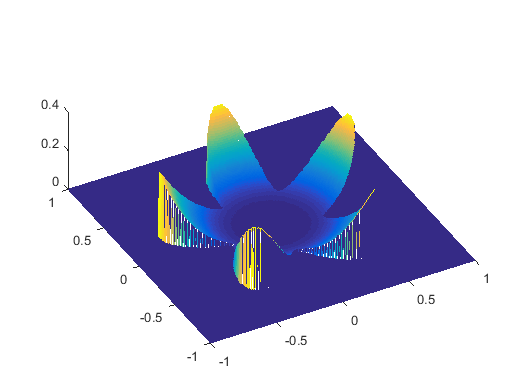}\\
  \caption{Example \ref{ExampleNitschesflower}: Left figure: a non-fitted triangulation of the interface. Right figure: approximate solution by method \eqref{femnonzeroflux}-\eqref{a_h}.}\label{figFlower}
\end{center}
\end{figure}
\begin{table}[!htp]

\footnotesize
\ra{1.1}
\begin{center}
\begin{tabular}{@{}l@{\hskip .2in}l@{\hskip .4in}c@{\hskip .4in}c@{\hskip .4in}c@{}}\toprule
$\rho^{-}$ & $\rho^{+}$ &   $e_h^0$    &   $e_{h,\rho}^1$  &$e_{h,\sqrt{\rho}}^1$\\ \midrule
1e+00      &  1e+01     &   1.6e-04    &   1.4e-02  &     1.1e-02  \\
1e-01      &  1e+02     &   1.7e-03    &   1.4e-02   &     3.4e-02 \\
1e-02      &  1e+03     &   1.7e-02    &   1.4e-02   &     1.1e-01 \\
1e-03      &  1e+04     &   1.7e-01    &   1.4e-02   &     3.4e-01 \\
1e-04      &  1e+05     &   1.7e-00    &   1.4e-02   &     1.1e+00 \\
 \bottomrule
  \end{tabular}
\end{center}
 \vskip3mm

\caption{Example \ref{ExampleNitschesflower}: errors and experimental orders of convergence (eoc) with  $\rho^-=1$, $\rho^+=10^5$, using method \eqref{femnonzeroflux}-\eqref{a_h}.}
\label{Table:Ex3_2}
\end{table}
}
\end{enumerate}

\appendix
\numberwithin{equation}{section}
\section{Proof Proposition \ref{propositiongeom}}\label{proofpropostiongeom}
Consider a node $z$ of the triangulation $\mathcal{T}_h$ and the patch of elements $\Delta_z$ (defined in Proposition \ref{proofpropostiongeom}) associated to $z$. Since we are assuming that $h$ is small enough and the interface is smooth we have that: the interface intersects each edge of triangulation at most once, or the interface coincides with an edge. Therefore, if $z\in \Omclosed$, there exists at least one node $z'\in \overline{\Delta}_{z}$ with $z'\in \Omclosed$, and the edge $\overline{e}$ connecting $z$ and $z'$ is completely contained in $\overline{\Omega}^{-}$. If there exists more than one node satisfying this property then it would follow that an element of the patch of $z$ is fully contained in $\Omega^{-}$. We then prove the remaining case. See Figure \ref{figureproposition} for an illustration of these definitions.
\begin{figure}[htbp]
\begin{center}
\includegraphics[scale=.5]{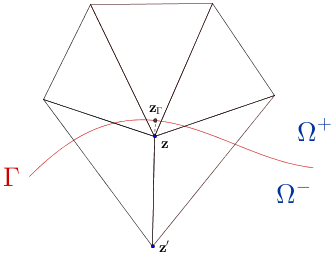}\\
\caption{Appendix \ref{proofpropostiongeom}: Illustration of definitions.}\label{figureproposition}
\end{center}
\end{figure}

Let $z_{\Gamma}$ the point on $\Gamma$ resulting of extending the segment $e$ connecting $z$ and $z'$. Denote by $A_1$ and $A_2$ the regions separated by the segment $e_{\Gamma} = \mathrm{Int}(\overline{z' z_{\Gamma}})$,  such that $A_1\cup A_2\cup e_{\Gamma} = \Delta_{z}\cap \Omega^{-}$. Then, by Lemma 4 in \cite{GSS15}, we have
\begin{equation*}
|e_{\Gamma}|^{2} \leq C \max\{|A_1|,|A_2|\}.
\end{equation*}
Thus, the result follows from $|e|\leq |e^{\Gamma}|$ and shape regularity.

\section{Proof of even extension}\label{proofofevenextension}
In this section we give an explicit construction of an even extension from $\Omega^+$ to $\Omega$. We apply this to continuous functions that are in $H^1(\Omega^+)$. The extension result \eqref{evenExtension} is a well-known result in partial differential equations. We sketch the proof with the aim to obtain  the explicit construction of the extension, which will be used in the proof of Lemma \ref{Discreteextensionlmm}.

For simplicity, assume that $\Omega^+$ is the inclusion (i.e. $\partial \Omega^+ = \Gamma$). For $\epsilon>0 $ define the tubular neighborhood
\begin{equation*}
R_{\epsilon}=\{ x \in \Omega: d(x):=\text{dist}(x, \partial \Omega^+) \le \epsilon\}.
\end{equation*}
Since $\Gamma$ is smooth, and for $\epsilon$ small enough, for each $x \in R_{\epsilon}$ there exists a unique point $x_{\partial\Omega^{+}} \in \partial \Omega^+$  such that
\begin{equation*}
|x-x_{\partial\Omega^{+}}|=\text{dist}(x,\partial \Omega^+).
\end{equation*}
Moreover, define the unit normal vector to $\partial \Omega^{+}$ pointing towards $x$ as: $n(x)=(x-x_{\partial\Omega^{+}})/|x-x_{\partial\Omega^{+}}|$.  Hence, for each $x \in R_{\epsilon}$ we define its ``reflection'' $\widetilde{x} = \widetilde{x}(x)$ by
\begin{equation*}
\widetilde{x} = x_{\partial\Omega^{+}}-d(x) n(x).
\end{equation*}

Now, given $v \in C(\Opclosed)$ we define $\widetilde{G} v  \in C(\Omega^+ \cup R_{\epsilon})$ such that:
\begin{equation*}
\widetilde{G}v(x)= \left\{
                 \begin{array}{ll}
                   v(x), & \hbox{if } x \in \Opclosed  \\
                   v(\widetilde{x}), & \hbox{if } x \in R_{\epsilon} \backslash \Omega^+.
                 \end{array}
               \right.
\end{equation*}
Our extension is complete by considering a cutoff function $\eta \in C_c( \Omega^+ \cup R_{\epsilon})$, such that $\eta  \equiv 1$ on $\Omega^+\cup R_{\epsilon/2}$. Of course, we extend $\eta$ to all of $\Omega$ by zero. Thus, the extension $G v \in C_c(\Omega)$ is defined  as follows
\begin{equation*}
Gv = \eta \widetilde{G}{v}.
\end{equation*}

The estimate is a well-known result.
\begin{equation}\label{evenExtension}
\|G v\|_{H^1(\Omega)} \le C \|v\|_{H^1(\Omega^+)}.
\end{equation}

\section{Proof of discrete extension Lemma \ref{Discreteextensionlmm}}\label{proofdiscreteextension}
We first introduce notation. For each node $x$ of the triangulation $\mathcal{T}_h$ we consider the patch associated to $x$, denoted by $\Delta_{x}$.  Moreover, we consider $\Delta^{-}_{x}$ as the restrictions to $\Omega^{-}$, i.e.,
\begin{equation*}
\Delta^{-}_{x} = \Delta_{x}\cap\Omega^{-}.
\end{equation*}
Considering definitions introduced in Appendix \ref{proofofevenextension}, we define the reflection of $\Delta_x^{-}$
\begin{equation*}
\widetilde{\Delta}^{-}_{x} = \{\widetilde{y}(y): y \in \Delta^{-}_{x}\} \subset \Opclosed.
\end{equation*}

Let $v\in V_{h}^{+}$, then, by \eqref{evenExtension} there exists $Gv\in H^{1}_0(\Omega)$ such that $Gv = v$ in $\Omega^{+}$, and
\begin{equation*}
\|Gv\|_{H^1(\Omega)}\leq C\|v\|_{H^1(\Omega^+)}.
\end{equation*}
We proceed constructing a stable interpolation operator of the extended function $Gv$ onto $V_h^c$, invariant on $V_{h}^{+}$. Define $P_1(Gv)$ for any node $x$ of the triangulation $\mathcal{T}_h$ by
\begin{equation*}
P_1(Gv)(x)= \left\{
                 \begin{array}{ll}
                   v(x), & \hbox{if } x \in \Opclosed,  \\
                   \frac{1}{|\Delta^{-}_{x}|} \int_{\Delta^{-}_{x}}Gv(y)dy , & \hbox{if } x \in \Omega^-,
                 \end{array}
               \right.
\end{equation*}
\revblue{i.e., $P_1$ preserves $v$ in $\overline{\Omega^{+}}$ and is the Cl\'ement interpolant (see \cite{MR0400739}) of $G(v)$ on  $\Omega^{-}$. Then, as a consequence of the definition of $P_1$ we have that}
\begin{equation}\label{P1estimate}
\|P_1(Gv)\|_{H^{1}(\Omega)}\leq C \|Gv\|_{H^{1}(\Omega)}.
\end{equation}
Observe that this definition does not guarantee that $P_1(Gv)$ coincides with $v$ for a node in $\Omega_h^{+}\backslash\Omega^{+}$. Then we need to correct the definition of the interpolant on these nodes. The needed extension is defined as follows
\begin{equation*}
Ev :=P_2(Gv)(x)= \left\{
                 \begin{array}{ll}
                   v(x), & \hbox{if } x \in \overline{\Omega}^{+}_h,  \\
                   \frac{1}{|\Delta^{-}_{x}|} \int_{\Delta^{-}_{x}}Gv(y)dy , & \hbox{if } x \in \Omega\backslash \overline{\Omega}^{+}_h.
                 \end{array}
               \right.
\end{equation*}
Notice that $P_2(Gv)$ and $P_1(Gv)$ agree at every node except the nodes $x$ in the following set
\begin{equation*}
S_h=\{ x\in\Omega_h^{+}\backslash\Omega^{+}: x \text{ is a node of  some triangle } T\in \mathcal{T}_{h}^{\Gamma} \}.
\end{equation*}
Define also the collection of triangles that have a node in $S$.
\begin{equation*}
\mathcal{M}_h=\{T \in \mathcal{T}_h: \text{ at least one of three vertices of } T \text{ belongs to } S_h\}.
\end{equation*}
If we define $e_h=P_2(Gv)-P_1(Gv)$ we see that
\begin{equation*}
\|\nabla e_h\|_{L^2(\Omega)}^2 = \sum_{T \in \mathcal{M}_h} \|\nabla e_h\|_{L^2(T)}^2.
\end{equation*}

For each $T \in \mathcal{M}_h$, let $x_T \in S$ be such that $|e_h(x_T)|=\text{max}_{y \in T} |e_h(y)|$. Then it is simple to show, using inverse estimates, that
\begin{equation*}
\|\nabla e_h\|_{L^2(T)}\le C \, |e_h(x_T)|.
\end{equation*}
By definition we have
\begin{equation*}
|e_{h}(x_T)| = \left|v(x_T)-\frac{1}{|\Delta^{-}_{x_T}|} \int_{\Delta^{-}_{x_T}}Gv(y)dy \right|= \left|v(x_T)-\frac{1}{|\Delta^{-}_{x_T}|} \int_{\widetilde{\Delta}^{-}_{x_T}}Gv(\widetilde{y})J(\widetilde{y})d\widetilde{y} \right|,
\end{equation*}
where $J(\widetilde{y})$ is the Jacobian of the map $\widetilde{y}(y) \rightarrow y$. Applying the definition of $Gv$ from the previous section, using that $v(\widetilde{x}_T)=v(x_T)$,  and the fact that $\int_{\widetilde{\Delta}^{-}_{x_T}}J(\widetilde{y})d\widetilde{y}=|\Delta^{-}_{x_T}|$ we get
\begin{equation*}
|e_{h}(x_T)| = \left|\frac{1}{|\Delta^{-}_{x_T}|} \int_{\widetilde{\Delta}^{-}_{x_T}}(v(\widetilde{x}_T)-v(\widetilde{y}))J(\widetilde{y})d\widetilde{y} \right|\leq C \, \text{diameter}\,(\widetilde{\Delta}^{-}_{x_T})\|\nabla v\|_{L^{\infty}(\widetilde{\Delta}^{-}_{x_T})}.
\end{equation*}
We see that $\widetilde{\Delta}^{-}_{x_T}\subset B_{\frac{d}{2}h}(T)=\{ y \in \Omega_h^+: \text{dist}(y,T)< \frac{d}{2}h \}$ for a $d$  large enough but independent of $h$ and $T$.

Using an inverse estimate we have
\begin{equation*}
\text{diameter}\,(\widetilde{\Delta}^{-}_{x_T}) \|\nabla v\|_{L^{\infty}(\widetilde{\Delta}^{-}_{x_T})}\leq C \|\nabla v\|_{L^{2}(B_{d\,h}(T))}.
\end{equation*}
Hence, we get
\begin{equation*}
\|\nabla e_h\|_{L^2(\Omega)}^2 \le C  \sum_{T \in \mathcal{M}_h} \|\nabla v\|_{L^{2}( B_{d\,h}(T))}^2 \le C \, \|\nabla v\|_{L^{2}(\Omega_h^+)}^2.
\end{equation*}
Finally, using the triangle inequality we get
\begin{equation*}
\|\nabla E v\|_{L^2(\Omega)}\le \|\nabla P_1(G v)\|_{L^2(\Omega)}+ C \|\nabla v\|_{L^{2}(\Omega_h^+)}.
\end{equation*}
The result now follows from using \eqref{P1estimate} and Poincare's inequality.

\bibliography{Bibliography_BGSS}
\bibliographystyle{plain}
%

\end{document}